\documentclass[11pt]{amsart}

\usepackage{lmodern}
\usepackage{mathrsfs}
\usepackage[english]{babel}
\usepackage{graphicx}
\usepackage{color}
\usepackage{enumerate}
\usepackage{mdwlist}
\usepackage{microtype}

\usepackage[all]{xy}
\usepackage{tikz-cd}
\usepackage{verbatim}
\usepackage{amsmath,amsfonts,amssymb,amsthm, amsxtra}
\usepackage{bm}
\usepackage{hyperref}

\usepackage{epsfig}
\usepackage[latin1]{inputenc}
\usepackage{enumerate}
\usepackage{verbatim}
\usepackage{amsfonts}

\usepackage{longtable}

\usepackage{tikz}
\usetikzlibrary{matrix}

\usepackage{hyperref}

\usepackage{enumitem}
\usepackage{cite}
\usepackage{geometry}
\usepackage{graphicx}
\usepackage{amssymb}
\usepackage{amsopn}
\usepackage{multirow}
\usepackage{multicol}
\usepackage{epstopdf}
\usepackage{setspace}
\usepackage{cite}

\newtheorem{thm}{Theorem}[section]
\newtheorem{cor}[thm]{Corollary}
\newtheorem{lem}[thm]{Lemma}
\newtheorem{prop}[thm]{Proposition}
\newtheorem{defn}[thm]{Definition}
\newtheorem{rmk}[thm]{Remark}

\newcommand{\RR}{\mathbb{R}}
\newcommand{\CC}{\mathbb{C}}
\newcommand{\ZZ}{\mathbb{Z}}
\newcommand{\QQ}{\mathbb{Q}}
\newcommand{\NN}{\mathbb{N}}

\newcommand{\PP}{\mathbb{P}}
\newcommand{\FF}{\mathbb{F}}
\newcommand{\HH}{\mathbb{H}}

\newcommand{\dc}{\mathcal{D}}
\newcommand{\lc}{\mathcal{L}}

\newcommand{\mc}{\mathcal{M}}
\newcommand{\qc}{\mathcal{Q}}

\newcommand{\la}{\langle}
\newcommand{\ra}{\rangle}
\newcommand{\bsm}{\left ( \begin{smallmatrix} }
\newcommand{\esm}{\end{smallmatrix} \right )}
\newcommand{\bma}{\begin{pmatrix}}
\newcommand{\ema}{\end{pmatrix}}
\newcommand{\opn}{\operatorname}
\newcommand{\op}{\oplus}
\newcommand{\fc}{\mathcal{F}}
\title{Boundary combinatorics of orthogonal modular 4-folds}
\author{Matthew Dawes}
\date{}
\begin{document}
\maketitle 
\def\thefootnote{}
\begin{abstract}
We study combinatorial problems related to the singularities and boundary components of toroidal compactifications of orthogonal modular varieties.
In particular, those associated with the moduli of algebraic deformation generalised Kummer 4-folds.
\end{abstract}
\section{Introduction}
Toroidal compactifications of modular varieties provide a rich source of combinatorics.
Two natural problems are to count boundary components, and to describe the singular locus.
Both of these problems have been studied in detail for Siegel modular 3-folds \cite{zbMATH00011135, zbMATH00520405}, but there are fewer results for orthogonal modular varieties.
In particular, while there are some results for orthogonal modular varieties of large dimension \cite{scattone}, less is known about those of small dimension.
The purpose of this article is to study these problems for orthogonal modular 4-folds $\fc_{2d}$ associated with the moduli of deformation generalised Kummer varieties of dimension 4, with polarisation of split type, and degree $2d$, where $d=p^2$ for an odd prime $p$.

In Theorem \ref{boundaryboundsthm} we produce bounds for the number of boundary curves of $\fc_{2p^2}$.
In Theorem \ref{fcFsingsthm} we bound the number of components of the singular locus of a toroidal compactification $\fc^{tor}_{2p^2}$ in the neighbourhood of a boundary curve.
To the best of the author's knowledge, these are the first such results for $\fc_{2d}$, and the first such results for orthogonal modular 4-folds.
\section{Notation}
\begin{tabular}{l l}
$( \bigoplus_i (\frac{*}{*}), \bigoplus_i C_{a_i} )$ & The finite quadratic form $\bigoplus_i (\frac{*}{*})$ (with values in $\QQ/2\ZZ$) on the \\
& abelian group $\op_i C_{a_i}$. \\ 
$\chi_X$ & The characteristic polynomial of the matrix $X$.\\
$C_r$ & The cyclic group of order $r$.\\
$\opn{div}(x)$ & The \emph{divisor} of $x \in M$ for a lattice $M$. \\
& (i.e. the positive generator of the ideal $(x, M)$.)\\  
$D(M)$ & The discriminant group of the (even) lattice $M$ \cite{Nikulin}. \\
$\dc_N$ & A connected spinor component of the domain  \\
& $\Omega_N = \{ [x] \in \PP(N \otimes \CC) \mid (x, x)= 0, (x, \overline{x})>0 \}$ \\
& for a lattice $N$ of signature $(2,n)$. \\
$\fc_{2d}$ & The orthogonal modular variety $\dc_{h^{\perp}} / \opn{O}^+(L, h)$ where $h \in L$ is primitive,\\
 & of split type, and degree $2d$.\\
$\widetilde{\Gamma}$ & The intersection $\Gamma \cap \widetilde{\opn{O}}(M)$ for $\Gamma \subset \opn{O}(M)$. \\
$\Gamma^+$ & The intersection $\Gamma \cap \opn{O}^+(M)$ for $\Gamma \subset \opn{O}(M)$. \\
$h^{\perp}$ & The orthogonal complement $h^{\perp} \subset L$ for $h \in L$. \\ 
$L$    & The lattice $U^{\op 3} \op \la -2(n+1) \ra$. \\
$L_{2x,2y}$ & The lattice $U^{\op 2} \op \la -2x \ra \op \la -2y \ra$. \\
$M^{\vee}$ & The dual of the lattice $M$.\\
$\opn{O}(L, h)$ & The group $\{g \in \opn{O}(L) \mid gh = h \}$ for $h \in L$. \\ 
$\opn{O}(M)$ & The orthogonal group of a lattice $M$. \\
$\opn{O}^+(M)$ & The kernel of the real spinor norm on $\opn{O}(M \otimes \RR)$. \\
$\widetilde{\opn{O}}(M)$ & The \emph{stable orthogonal group} $\widetilde{\opn{O}}(M) = \{ g \in \opn{O}(M) \mid g_{\vert D(M)}=id \}$. \\
$\phi$ & The Euler $\phi$-function.\\
$\Phi_r$ & The $r$-th cyclotomic polynomial.\\
\end{tabular}
\section{Moduli of Deformation generalised Kummer varieties}
Let $A$ be an abelian surface, and let $A^{[n+1]}$ be the Hilbert scheme parametrising $(n+1)$-points on $A$.
The Hilbert scheme $A^{[n+1]}$ inherits an addition from $A$, and so there is a projection 
\begin{equation*}
p:A^{[n+1]} \rightarrow A.
\end{equation*}
The fibre $p^{-1}(0)$ is an irreducible symplectic (compact hyperk\"ahler) manifold known as a \emph{generalised Kummer variety} \cite{beauvillesexamples}.
A deformation $X$ of $p^{-1}(0)$ is an irreducible symplectic manifold known as a \emph{deformation generalised Kummer variety}.

There is a lattice structure $L$ on $H^2(X, \ZZ)$ defined by the Beauville-Bogomolov-Fujiki \cite{beauvillesexamples} form.
By the results of Rapagnetta \cite{rapagnetta2}, $L$ is equal to 
\begin{equation*}
L=U^{\op 3} \op \la -2(n+1) \ra,
\end{equation*}
where $\la -2(n+1) \ra$ is the rank 1 lattice generated by a vector of square length $-2(n+1)$ and  $U$ is the hyperbolic plane.
A basis $\{e, f\}$ for $U$ is said to be \emph{standard} or a \emph{standard basis for $U$} if its  Gram matrix is given by
\begin{equation*}
\begin{pmatrix}
0 & 1 \\
1 & 0 \\
\end{pmatrix}.
\end{equation*}

A choice of ample line bundle $\lc \in \opn{Pic}(X)$ defines a \emph{polarisation} for $X$.
The first Chern class of $\lc$ defines a vector $h:=c_1(\lc) \in L$.
The \emph{degree} $2d$ of $\lc$ is defined as the square length $h^2$, and the \emph{polarisation type} of $\lc$ is defined as the $\opn{O}(L)$-orbit of $h$.
We assume all polarisations are \emph{primitive}; that is, the vector $h$ is primitive in the lattice $L$.

Let $\opn{O}^+(L, h)$ be the subgroup of $\opn{O}(L, h)$ consisting of all elements of spinor norm 1, and let $\dc_{h^{\perp}}$ be a connected component of the quadric
\begin{equation*} 
\Omega_{h^{\perp}}=\{[x] \in \PP(h^{\perp} \otimes \CC) \mid \text{$(x,x)=0$, $(x, \overline{x})>0$} \}
\end{equation*}
preserved by the kernel of the real spinor norm on $\opn{O}(h^{\perp} \otimes \RR)$.

There is a GIT quotient $\mc$ parametrising deformation generalised Kummer varieties of fixed dimension and polarisation type.
By Theorem 3.8 of \cite{Handbook}, for every connected component $\mc'$ of $\mc$ there is a finite-to-one dominant morphism $\psi: \mc' \rightarrow \fc$ to an orthogonal modular variety $\fc$ (i.e. a quotient of a Hermitian symmetric domain of type IV by an arithmetic subgroup of $\opn{O}(2,m)$).
Here, the orthogonal modular variety $\fc$ is given by 
\begin{equation*}
\fc= \opn{O}^+(L, h) \backslash \dc_{h^{\perp}}.
\end{equation*}
\begin{prop}\label{orbitclass}
Suppose $h \in L$ is primitive of length $2d >0$ with $\opn{div}(h) = f$.
Let $g = \left (2(n+1)f^{-1}, 2df^{-1} \right )$, 
$w=(g,f)$, 
$g = w g_1$, 
and $f = wf_1$. 
Then 
$2(n+1) = fgn_1 = w^2 f_1 g_1 n_1$ 
and 
$2d=f g d_1 = w^2 f_1 g_1 d_1$, 
where 
$(n_1, d_1) = (f_1, g_1) = 1$.
\begin{enumerate}
\item If $g_1$ is even, then $h$ exists if and only if $(d_1, f_1) = (f_1, n_1) = 1$ and $d_1/n_1$ is a quadratic residue modulo $f_1$. 
Moreover, the number of $\widetilde{\opn{O}}(L)$-orbits of $h$ with fixed $f$ is equal to $w_+(f_1) \phi(w_-(f_1)).2^{\rho(f_1)}$, where $w=w_+(f_1) w_-(f_1)$, $w_+(f_1)$ is the product of all powers of primes dividing $(w, f_1)$, $\rho(n+1)$ is the number of prime factors of $n+1$, and $\phi$ is the Euler function.
\item If $g_1$ is odd and $f_1$ is even, or $f_1$ and $d_1$ are both odd, then such an $h$ exists if and only if $(d_1, f_1) = (t_1, 2f_1)=1$ and $-d_1 / n_1$ is a quadratic residue modulo $2f_1$. 
The number of $\widetilde{\opn{O}}(L)$-orbits is equal to $w_+(f_1) \phi(w_-(f_1)).2^{\rho(f_1/2)}$ if $f_1$ is even, and to $w_+(f_1) \phi(w_-(f_1)).2^{\rho(f_1)}$ if $f_1$ and $d_1$ are both odd.
\item If $g_1$ and $f_1$ are both odd and $d_1$ is even, then such an $h$ exists if and only if $(d_1, f_1) = (n_1, 2f_1)=1$, $-d_1/(4t_1)$ is a quadratic residue modulo $f_1$, and $w$ is odd.
In such a case, the number of $\widetilde{\opn{O}}(L)$-orbits of $h$ is equal to $w_+(f_1) \phi(w_-(f_1)).2^{\rho(f_1)}$.
\item If $c \in \ZZ$ (determined modulo $f$) satisfies $(c, f)=1$ and $b = (d+c^2(n+1))/f^2$, then
\begin{equation*}
h^{\perp} \cong 2U \oplus B,
\end{equation*}
where 
\begin{equation*}
B = \begin{pmatrix} -2b & c \frac{2(n+1)}{f} \\ c \frac{2(n+1)}{f} & -2t \end{pmatrix}.
\end{equation*} 
\end{enumerate}
\end{prop}
\begin{proof}
Identical to Proposition 3.6 of \cite{ModuliSpacesOfIrreducibleSymplecticManifolds}.
(Noting that the Beauville lattice of a deformation generalised Kummer variety and an irreducible symplectic manifold of $K3^{[2(n+1)]}$-type differ by a factor of $2E_8(-1)$.) 
\end{proof}
\begin{defn}
A polarisation determined by a primitive vector $h \in L$ is said to be \emph{split}, or $h$ is said to be \emph{split}, if $\opn{div}(h)=1$.
\end{defn}
\begin{cor}
If $h \in L$ is split then,
\begin{enumerate}
\item the polarisation type of $h$ is uniquely determined by the length $h^2$;
\item the lattice $h^{\perp}$ is isomorphic to $L_{2(n+1), 2d}:=2U \op \la -2(n+1) \ra \op \la -2d \ra$.
\end{enumerate}
\end{cor}
\begin{defn}
Let $\fc_{2d}$ denote the modular variety $\fc_{2d} = \opn{O}^+(L, h) \backslash \dc_{h^{\perp}}$ where $h \in L$ is split and $h^2 = -2d$.
\end{defn}
\section{The group $\opn{O}(L, h)$}
From now on, we assume all polarisations are split and fix $2(n+1)=6$.
Throughout, we shall use $h_x$ to denote split $h \in L$ of degree $x$.
\begin{prop}\label{modgroup}
If $h=h_{2d}$ and $d > 2$, then 
\begin{equation*}
\opn{O}(L, h) \cong \{ g \in \opn{O}(L_{6, 2d}) \mid g v^* = v^* + L_{6,2d} \},
\end{equation*}
where $v$ generates the $\la -2d \ra$ factor of $L_{6,2d}$ and $v^* = (2d)^{-1} v \in L^{\vee}$.
Moreover, if $d=p^2$ for an odd prime $p$, then $\opn{O}(L, h) \leq \opn{O}(L_{6, 2})$.
\end{prop}
\begin{proof} 
(c.f. Proposition 3.12  of \cite{ModuliSpacesOfIrreducibleSymplecticManifolds}.)
By noting that $\opn{O}(L, h)$ acts on both $\la h \ra$ and $\la h \ra^{\perp}$, but trivially on $\la h \ra$,  we can immediately identify $\opn{O}(L, h)$ with a subgroup of $\opn{O}(L_{6,2d})$.

There is an inclusion of abelian groups
\begin{equation*}
L/(\la h \ra \op \la h \ra^{\perp})  \subset \la h \ra^{\vee} / \la h \ra \op (\la h \ra^{\perp})^{\vee} / (\la h \ra^{\perp}) = D(\la h \ra) \op D(\la h \ra^{\perp})
\end{equation*}
defined by the series of overlattices
\begin{equation*}
\la h \ra \oplus h^{\perp} \subset L \subset L^{\vee} \subset \la h^{\vee} \ra \oplus (h^{\perp})^{\vee},
\end{equation*}
and so the isotropic subgroup $H = L / ( \la h \ra \oplus h^{\perp})$ can be regarded as a subgroup of $D(\la h \ra) \op D( \la h \ra^{\perp})$. 
Let $p_h$ and $p_{h^{\perp}}$ denote the  corresponding projections $p_h: H \rightarrow D(\la h \ra)$, $p_{h^{\perp}}:H \rightarrow D(\la h \ra^{\perp})$.
By Proposition \ref{orbitclass}, we can assume that $h$ is given by $h = e_3 + d f_3 \in U \oplus \la -6 \ra$.
Let $k'_1 = (2d)^{-1}k_1$, $k'_2 = 6^{-1}k_2$, $k'_3 = (2d)^{-1}h$, and $k_1 = e_3 - df_3$,  where $k_2$ is a generator of the $\la -6 \ra$ factor of $L$.
Take a basis $\{e_1, f_1, e_2, f_2, k'_1, k'_2 \}$ for $(h^{\perp})^{\vee}$.
By direct calculation, $H = \la k_3' - k'_1, d(k_1' + k'_3) \ra +  \la h \ra \oplus h^{\perp}$, $p_{h^{\perp}}(H) = \la k_1' \ra$, and $D(h^{\perp}) = \la k'_1 \ra \oplus \la k'_2 \ra$.
By applying Corollary 1.5.2 of \cite{Nikulin}, 
\begin{equation*}
\opn{O}(L, h) \cong \{ g \in \opn{O}(h^{\perp}) \mid  g \vert _{p_{h^{\perp}}(H)} = \opn{id} \}, 
\end{equation*}
and the first part of the claim follows.

For the second part of the claim, embed $L_{6, 2p^2} \subset L_{6,2}$ by identifying factors of $2U \op \la -6 \ra$ and mapping
\begin{equation*}
L_{6,2p^2} \ni t + ak_1 \mapsto t + bu + apk \in L_{6,2},
\end{equation*} 
where $t \in 2U \op \la -6 \ra$, $k$ generates $\la -2 \ra \subset L_{6,2}$, and $a, b \in \ZZ$.
Define the totally isotropic subspace $M \subset D(L_{6,2p^2})$ by $M = L_{6, 2} / L_{6, 2p^2} \subset D(L_{6, 2p^2})$. 
By the above, if $g \in \opn{O}(L, h)$, then $g(k_1') = k_1'+ L_{6,2p^2}$. 
As $M \subset \la k_1' \ra + L_{6,2p^2} \subset D(L_{6,2p^2})$ and $g(L_{6,2p^2}) = L_{6,2p^2}$, then $g$ extends to a unique element of $\opn{O}(L_{6,2})$.
\end{proof}
Let $p$ be an odd prime.
We use an idea in \cite{Kondo} (who attributes it to O'Grady) to bound the index $\vert \opn{O}(L_{6, 2}) : \opn{O}(L_{6}, h_{2p^2}) \vert$.
This involves considering the finite quadratic space $\qc_p$, defined by
\begin{equation*}
\qc_p  = L_{6,2}/pL_{6,2} \subset L_{6,2p^2} / pL_{6,2},
\end{equation*}
and a number of classical results on orthogonal groups of finite type
(which can be found in \cite{Dieudonne}, but are stated below for the convenience of the reader).

For $i \in \NN$, let $H_i$ denote hyperbolic planes over $\FF_p$, and let $V_{\theta}$ denote the quadratic space $\la u, v \ra$ whose bilinear form is given by $(u,u)=1$, $(u, v)=0$, and $(v, v) = \theta$ where $-\theta \notin (\FF_q^*)^2$.

A non-degenerate quadratic space $V$ over a finite field $\FF_q$ of odd order $q$ is uniquely determined by  $\opn{dim} V$ and the discriminant $\Delta = \opn{det }B \in \FF_q^* / (\FF_q^*)^2$, where $B$ is the bilinear form on $V$. 

If $\opn{dim} V=2m$ and $\epsilon= (-1)^m \Delta \in \FF_q^*/(\FF_q^*)^2$, then $V$ is isomorphic to 
\begin{equation*}
\begin{cases} 
V_{\epsilon}^{2m}  = H_1 \oplus \ldots \oplus  H_m & \text{if $\epsilon =1$} \\
 V_{\epsilon} ^{2m} = V_{\theta} \oplus H_1 \oplus H_2  \oplus \ldots \oplus H_{m-1} & \text{if $\epsilon = -1$}.
\end{cases}
\end{equation*}
If $\opn{dim} V = 2m+1$, there is a single isomorphism class for $V$ given by 
\begin{equation*}
V^{2m+1}=H_1 \op \ldots H_m \op \la \theta \ra
\end{equation*}
for $0 \neq \theta \in \FF_q$.

We also need to know the order of $\opn{O}^+(V)$. 
As in \cite{Dieudonne},
\begin{equation*}
\begin{cases}
\vert \opn{O}^+(V^{2m+1})\vert = (q^{2m} -1) q^{2m-1} (q^{2m-2} - 1) \ldots (q^2 - 1) q \\
\vert \opn{O}^+(V_{\epsilon}^{2m}) \vert = (q^{2m-1} - \epsilon q^{m-1})(q^{2m-2}  - 1) q^{2m-3} \ldots (q^2 - 1 ) q.
\end{cases}
\end{equation*}
\begin{lem}\label{translemma} 
If $u, v \in \qc_p$ and $u^2 = v^2 \in \FF_p^*/(\FF_p^*)^2$ for $p>3$, then $u$ and $v$ are equivalent under $\opn{O}(L_{6,2})$.
\end{lem}
\begin{proof}
Let $\{e_1, f_1, e_2, f_2, v_1, v_2 \}$ be a basis for $L_{6,2}$ where $v_1$, $v_2$ are the respective generators of $\la -6 \ra$, $\la -2 \ra$  and $\{e_i, f_i \}$ are standard bases for the two copies of $U$. 

We define elements of $\opn{O}(L_{6,2})$.
For $e \in L_{6,2}$ isotropic and $a \in e^{\perp} \subset L_{6,2}$ there are elements $t(e,a) \in \opn{O}(L)$ (known as \emph{Eichler transvections}) defined by
\begin{equation*}
t(e, a) : v \mapsto v - (a,v)e + (e,v)a - \frac{1}{2}(a,a)(e,v)e
\end{equation*}
(\cite{eichler} or \S3 of \cite{abelianisation}).
The action of $t(e_2, v_1)$ and $t(e_2, v_2)$ on $w = (w_1, w_2, w_3, w_4, w_5, w_6) \in L_{6,2}$ are given by  
\begin{equation*}
\begin{cases}
t(e_2, v_1): w \mapsto (w_1, w_2, w_3 + 3w_4 + 6w_5, w_4, w_5 + w_4, w_6  ) \\
t(e_2, v_2): w \mapsto (w_1, w_2, w_3 + w_4 + 2w_6, w_4, w_5, w_6 + w_4).
\end{cases}
\end{equation*}
One can also obtain elements of $\opn{O}(L_{6,2})$ by the trivial extension of elements in $\opn{O}(2U)$ for an embedding $2U \subset L_{6,2}$.
In particular, if $(w,x,y,z)$ is taken on the standard basis $\{e_i, f_i\}$ of $U \op U$ then the map
\begin{equation}\label{sliso}
U \op U \ni (w,x,y,z) \mapsto 
\begin{pmatrix}
 w & -y \\ 
z & x 
\end{pmatrix},
\end{equation}
identifies $M_2(\ZZ)$ with $U \op U$ (where the inner product on $M_2(\ZZ)$ is defined by $\opn{det}$). 
An element $(A, B) \in \opn{SL}(2, \ZZ) \times \opn{SL}(2, \ZZ)$ defines an element in $\opn{O}(U \op U)$ by the mapping
\begin{equation*}
(A, B) : 
\begin{pmatrix}
w & -y \\ 
z & x 
\end{pmatrix} 
\mapsto 
A 
\begin{pmatrix} 
w & -y \\ 
z & x 
\end{pmatrix}
B^{-1}.
\end{equation*}

Let $x = (x_1, x_2, x_3, x_4, x_5, x_6) \in L_{6,2} / p L_{6,2}$ be non-zero. 
We can assume that $x_4 \neq 0$ by (if required) applying $t(e_2, v_1)$ or $t(e_2, v_2)$, or permuting $\{x_1, x_2, x_3, x_4 \}$ by elements in $\opn{O}(2U)$. 
By rescaling $x$ so that $x_4 = 1$, and by repeated application of $t(e_2, v_1)$ and $t(e_2, v_2)$, the element $x$ can be transformed to an element of the form $(x_1',x_2',x_3',x_4', 0, 0)$, and so can be identified with an element of $2U$. 
Because of the existence of a Smith normal form for the associated matrix \eqref{sliso}, $x$ can be mapped to an element of the form $(r, s, 0,0,0,0)$ by using the image of $\opn{SL}(2, \ZZ) \times \opn{SL}(2, \ZZ)$ in $\opn{O}(2U)$.
One can assume $x$ is of the form $(1, a, 0,0,0,0)$ by rescaling (if necessary).

Now suppose $u, v \in L_{6,2} / p L_{6,2}$ are given by $u=(1, a, 0,0,0,0)$ and $v=(1, b, 0,0,0,0)$.
If $ab^{-1} \in (\FF_p^*)^2$, then there exists $\mu, \lambda \in \FF_p$ such that $(\mu u)^2 = (\lambda v)^2$.  
Define $\hat{u}$ and $\hat{v}$ by $\hat{u}=\mu u = (u_1, u_2, 0,0,0,0)$, $\hat{v}=\lambda v = (v_1, v_2, 0,0,0,0)$ and suppose (without loss of generality) that $\hat{u} - \hat{v} = (r,s,0,0,0,0)$ is non-zero.
Take representatives for $r,s$ modulo $p$, and let
\begin{equation*}
d = 
\begin{cases}
r & \text{if $s=0$}\\
s & \text{if $r=0$} \\
\opn{gcd}(r,s) & \text{otherwise}.
\end{cases}
\end{equation*}
If $r_1, r_2, s_1, s_2$ are solutions to $r_2 u_1 + r_1 u_2 = d$ and  $s_2 v_1 + y_2 v_2 = d$ taken modulo $p$, define the elements $u', v', w \in e_1^{\perp} \cap f_1^{\perp} \subset L_{6,2}$  by $u'=(r_1, r_2, 0,0,0,0)$, $v'=(s_1, s_2, 0,0,0,0)$, and $w=(d^{-1}r, d^{-1}s, 0,0,0,0)$ where $r' \equiv r \bmod{p}$ and $s' \equiv s \bmod{p}$.
Then, over $\FF_p$, $(\hat{u}, u') = d$, $(\hat{v}, v) = d$, and $t(e_2, v')t(f_2, w)t(e_2, u'): \hat{u} \mapsto \hat{v}$.
The result follows. 
\end{proof}
\begin{thm}\label{finindexthm}
Let $p>3$ be prime.
Then $\opn{O}^+(L, h_{2p^2})$ is of finite index in $\opn{O}^+(L , h_2)$ and   
$\vert  \opn{O}^+(L , h_2) : \opn{O}^+(L , h_{2p^2}) \vert \leq 16 (p^5 + p^2)$.
\end{thm}
\begin{proof}
There is a natural homomorphism $\opn{O}(L_{6,2}) \rightarrow \opn{O}(L_{6,2}/pL_{6,2})$.
If $v, w \in \qc_p$ and $v^2 = w^2 \bmod{ (\FF_p^*)^2}$ then, by Lemma \ref{translemma}, $v \sim w$ under the action of $\opn{O}(L_{6,2})$ and so $v \sim w$ under the action of $\opn{O}(L_{6,2}/pL_{6,2})$.
The group $\opn{O}(L, h_{2p^2}) \subset \opn{O}(L_{6,2})$ stabilises a hyperplane $\Pi \subset Q_p$ and, as $\opn{Stab}_{\opn{O}(L_{6,2} /p L_{6,2})}(\Pi) = \opn{O}(L_{6,2p^2} /p L_{6,2})$, by the orbit-stabiliser theorem, 
\begin{equation*}
\vert \opn{O}(L_{6,2}) : \opn{Stab}_{ \opn{O}(L_{6,2})} (\Pi) \vert = \vert \opn{O}(L_{6,2} /p L_{6,2}) : \opn{O}(L_{6,2p^2} /p L_{6,2}) \vert
\end{equation*}
and
\begin{equation*}
\vert \opn{O}^+(L_{6,2}) : \opn{Stab}_{ \opn{O}^+(L_{6,2})} (\Pi) \vert = \vert \opn{O}^+(L_{6,2} /p L_{6,2}) : \opn{O}^+(L_{6,2p^2} /p L_{6,2}) \vert.
\end{equation*}
By Lemma~\ref{modgroup}, $\opn{O}(L, h_{2p^2}) \subset \opn{O}(L_{6,2})$ and so, 
\begin{equation*}
\widetilde{\opn{O}}^+(L_{6,2p^2}) \leq \opn{O}^+(L, h_{2p^2}) \leq \opn{Stab}_{ \opn{O}^+(L_{6,2})}(\Pi) \leq \opn{O}^+(L_{6,2p^2}).
\end{equation*}
As $\opn{O}(D(L_{6,2p^2})) \cong C_2^{\op 3}$, then
\begin{equation}\label{stabindex}
\vert \opn{Stab}_{ \opn{O}(L_{6,2})}(\Pi) : \opn{O}(L, h_{2p^2}) \vert \leq \vert \opn{O}(L_{6,2p^2}) : \widetilde{\opn{O}}(L_{6,2p^2}) \vert = 8,
\end{equation}
and so,
\begin{align*}
 \vert \opn{O}^+(L_{6,2}) : \opn{O}^+(L, h_{2p^2}) \vert 	
& \leq 8 \vert \opn{O}^+(L_{6,2} /p L_{6,2}) : \opn{O}^+(L_{6,2p^2} /p L_{6,2}) \vert \\
& \leq 8 \frac{(p^5 - \epsilon p^2)(p^4 -1)p^3(p^2 -1)p }{(p^4-1)p^3(p^2 -1) p} \\
& \leq 8 (p^5 + p^2).
\end{align*}
\end{proof}
\section{Toroidal compactifications}
\subsection{Overview}
Toroidal compactifications are defined fully in \cite{AMRT}.
We state only the results we require, following  \cite{Handbook}.
Let $\fc$ denote the orthogonal modular variety $\dc_M/\Gamma$, where $M$ is a lattice of signature $(2, n)$ and (without loss of generality) $\Gamma$ is a neat normal subgroup of $\opn{O}^+(M)$.
There is a partial compactification $\dc_M^*$ of $\dc_M$ (taken in the compact dual $\dc_M^{\vee}$) to which the action of $\Gamma$ extends.
The compactification $\dc_M^*$ admits a decomposition  
\begin{equation}\label{dm*dec}
\dc_M^* = \dc_M \sqcup \bigsqcup_{\Pi} F_{\Pi} \sqcup \bigsqcup_{\ell} F_{\ell},
\end{equation}
where the rational boundary components $F_{\Pi}$ (\emph{boundary curves}) and $F_{\ell}$ (\emph{boundary points}) are symmetric spaces corresponding to $\Gamma$-equivalence classes of totally isotropic planes and isotropic lines in $M \otimes \QQ$. 

For a rational boundary component $F$ (given by some $F_{\ell}$ or $F_{\Pi}$ in \eqref{dm*dec}), let $N(F) \subset \opn{O}(2, n)$ be the stabiliser of $F$, $W(F)$ be the unipotent radical of $N(F)$, and $U(F)$ be the centre of $W(F)$.
Denote the intersections of $N(F)$, $U(F)$, and $W(F)$ with $\Gamma$ by $N(F)_{\ZZ}$, $U(F)_{\ZZ}$, and $W(F)_{\ZZ}$.
Let $\dc_M(F)$ denote the domain $U(F). \dc_M \subset \dc_M^{\vee}$.
Because of the Langlands decomposition for the parabolic subgroup $N(F)$, the domain $\dc_M(F)$ decomposes as  
\begin{equation*}
\dc_M(F) = F \times V(F) \times U(F)_{\CC},
\end{equation*}
where $V(F)$ is the complex vector space $W(F)/U(F)$.

If $\dc_M(F)'$ is the quotient $\dc_M(F)'=\dc_M(F)/U(F)_{\CC}$, then the spaces $\dc_M(F)$, $\dc_M(F)'$, and $F$ are related by the diagram
\begin{equation*}
\begin{tikzpicture}
\matrix (m) [matrix of math nodes,row sep=3em,column sep=4em,minimum width=2em]
{
\mathcal{D}_M(F) & \\
& \mathcal{D}_M(F)'\\
F & \\
};
\path[-stealth]
(m-1-1) edge node [above] {$\pi_F'$}  (m-2-2) 
(m-2-2) edge node [above] {$p_F$} (m-3-1)
(m-1-1) edge node [left] {$\pi_F$} (m-3-1)
;
\end{tikzpicture}
\end{equation*}  
where $\pi_F$, $p_F$, and $\pi_F'$ are the natural projections onto $F$, $F$, and $\dc_M(F)$, respectively. 

The space 
\begin{equation}\label{phom}
\pi_F':\dc_M(F) \rightarrow \dc_M(F)'
\end{equation}
is a principal homogeneous space for $U(F)_{\CC}$ and admits an $N(F)_{\ZZ}$-action.
By taking the quotient of \eqref{phom} by $U(F)_{\ZZ} \subset N(F)_{\ZZ}$, one obtains the principal fibre bundle
\begin{equation}\label{fbundle}
\dc_M(F) / U(F)_{\ZZ} \rightarrow \dc_M(F)'
\end{equation}
whose fibre is equal to the algebraic torus $T(F):=U(F)_{\CC} / U(F)_{\ZZ}$.  

There is a real cone $C(F) \subset U(F)$.
By taking a fan $\Sigma$ in the closure of $C(F)$ and then replacing the torus $T(F)$ in the fibre bundle \eqref{fbundle} with the toric variety $X_{\Sigma(F)}$ one obtains a new bundle over $\dc_M(F)$ with fibre $X_{\Sigma(F)}$.

One constructs a partial compactification $\fc(F)^{tor}$ for $\fc$ in a neighbourhood of $F$ by taking the closure of $\dc_M/U(F)_{\ZZ}$ in the new bundle, and then taking the quotient by $N(F)_{\ZZ}$.
A toroidal compactification for $\fc$ is obtained by taking the set of partial compactifications over all rational boundary components $F$ and gluing by identifying the copies of $\fc$ contained in each one. 
\subsection{Invariants associated with $F$}
\begin{defn}
If $M$ is a lattice and $E \subset M$ is a primitive totally isotropic sublattice, let $H_E := E^{\perp \perp}/E \subset D(M)$ where $E^{\perp \perp} \subset M^{\vee}$. 
\end{defn}
\begin{lem}\label{isobasisp}
Let $E \subset L_{6,2p^2}$ be a primitive totally isotropic sublattice of rank $2$ corresponding to the boundary component $F$. 
Then there exists a $\ZZ$-basis $\{v_1, \ldots, v_6 \}$ of $L_{6,2p^2}$ so that $\{v_1, v_2\}$ is a basis for $E$ and $\{v_1, \ldots, v_4 \}$ is a basis for $E^{\perp}$. 
Furthermore, the basis can be chosen so that the Gram matrix
\begin{equation}\label{eqnform}
Q 
= 
((v_i, v_j)) 
= 
\begin{pmatrix} 
0 & 0 & A \\ 
0 & B & C \\ 
{}^tA & {}^tC & D 
\end{pmatrix}
\end{equation}
where
\begin{equation*}
A = 
\begin{pmatrix} 
a_1 & 0 \\ 
0 & a_1 a_2 
\end{pmatrix},
\end{equation*}
$a_1$, $a_2$ are the elementary divisors of the group $D(L_{6,2p^2}) / H_E^{\perp}$, and $B$ is the quadratic form on $E^{\perp} /E$.

Moreover, 
\begin{equation*}
  (a_1, a_1a_2) \in \{ (1,1), (1,p), (1,2p) \}.
\end{equation*}
\end{lem}
\begin{proof}
As the lattices $E$ and $E^{\perp}$ are primitive in $L_{6,2p^2}$, the existence of a basis on which $Q$ assumes the form of \eqref{eqnform} is immediate. 
The Smith normal form of the matrix $A$ embeds $\la v_5, v_6 \ra$ in the dual $\la v_5^*, v_6^* \ra$.
Therefore, the elementary divisors of $A$ correspond to the elementary divisors of the abelian group $\la v_5^*, v_6^* \ra / \la v_5, v_6 \ra$ (c.f. \cite{GHSK3}). 
  
If $H_E = E^{\perp \perp}/E \subset D(L_{6,2p^2})$, then  $H_E^{\perp} + L_{6,2p^2} = \la v_1^*, \ldots, v_4^* \ra$ in $D(L_{6,2p^2})$, and so  $\la v_5^*, v_6^* \ra / \la v_5, v_6 \ra \cong D(L_{6,2p^2}) / H_E^{\perp}$. 

As $E$ is totally isotropic in $L_{6,2p^2}$, then $H_E$ is totally isotropic in $D(L_{6,2p^2})$. 
If $D(L_{6,2p^2})$ is identified with $((-1/6) \op (-1/2p^2), C_6 \op C_{2p^2})$, then $(x, y) \in D(L_{6,2p^2})$ is isotropic if and only if
\begin{equation*}
p^2x^2 + 3y^2 = 0 \mod{12 p^2}.
\end{equation*}
As $(3,p)=1$ then $p \vert y$, $p^2x^2 + 3p^2y_1^2 = 0 \mod{12p^2}$, and $x^2 + py_1^2 = 0 \mod{6}$.
  
By considering squares modulo 6, we conclude that $x = 0$ or $3$, and that $x$ and $y$ must have equal parities. 
Therefore, the isotropic elements of $D(L_{6,2p^2})$ are given by the set 
\begin{equation*}
  (x, y) \in \{ (0,2kp), (3, (2k+1)p) \mid k \in \ZZ \};
\end{equation*}
the primitive isotropic subspaces of rank 1 in $D(L_{6,2p^2})$ are generated by $x_1 := (0, 2p)$ and $x_2 := (3, p)$;
and the single primitive totally isotropic subspace of rank 2 is generated by $\la x_1, x_2 \ra$. 

If $H_E = \la x_1 \ra$ then 
\begin{equation*}
  H_E^{\perp} = \{ (a,b) \in D(L_{6,2p^2}) \mid 6b \equiv 0 \mod{6p} \},
\end{equation*}
and so, $p \vert b$ and $H_E^{\perp} = \la (1,0), (0,p) \ra \cong C_6 \op C_{2p}$.

If $H_E = \la x_2 \ra$ then 
\begin{equation*}
  H_E^{\perp} = \{ (a,b) \in D(L_{6,2p^2}) \mid pa + b \equiv 0 \mod{2p} \},
\end{equation*}
and so, 
$p \vert b$, 
$2 \vert (a+b)$,
and $H_E^{\perp} = \la (1,p), (2,0) \ra$. 

If $y_1 = (1,p)$ and $y_2 = (2,0)$, we also have the relations 
$6py_1 = 0$,
$3y_2 = 0$,
and $p(2y_1-y_2) = 0$. 
As $p \equiv \pm 1 \bmod{6}$, then 
$2py_1 = \pm y_2$  
and so $H_E^{\perp} = \la y_1 \ra = \la (1,p) \ra \cong C_3 \op C_{2p}$. 

If $H_E = \la x_1, x_2 \ra$, then $H_E^{\perp} = \la y_1 \ra = \la (1,p) \ra \cong C_3 \op C_{2p}$. 

We conclude that,
\begin{enumerate}
\item if $H_E = \{0 \}$, then $H_E^{\perp}  = D(L_{6,2p^2})$ and $D(L_{6,2p^2})/H_E^{\perp} \cong \{0\}$; 
\item if $H_E = \la x_1 \ra$, then $H_E^{\perp} = \la (1,0), (0,p) \ra \cong C_6 \op C_{2p}$ and $D(L_{6,2p^2})/H_E^{\perp} \cong  C_p$; 
\item if $H_E = \la x_2 \ra$, then $H_E^{\perp} = \la (1,p) \ra \cong C_3 \op C_{2p}$ and $ D(L_{6,2p^2})/H_E^{\perp} \cong C_2 \op C_{p}$; 
\item if  $H_E = \la x_1, x_2 \ra$, then $H_E^{\perp} = \la (1,p) \ra \cong C_3 \op C_{2p}$ and $D(L_{6,2p^2})/H_E^{\perp} \cong C_2 \op C_{p}$,
\end{enumerate}
and the result follows.
\end{proof}
\begin{lem}\label{Qbasis}
There exists a basis $\{v_1, \ldots, v_6\}$ for $L_{6,2p^2} \otimes \QQ$ so that $\{v_1, v_2\}$ is a $\ZZ$-basis for $E$, $\{v_1, \ldots, v_4\}$ is a $\ZZ$-basis for $E^{\perp}$, and 
\begin{equation*}
  Q = ((v_i, v_j)) =
  \begin{pmatrix} 
    0 & 0 & A \\ 
    0 & B & 0 \\ 
    A & 0 & 0 
  \end{pmatrix}
\end{equation*}
where $A$ and $B$ are as in Lemma \ref{isobasisp}.
\end{lem}
\begin{proof} 
(Essentially as in Lemma 2.24 of \cite{GHSK3}.) 
  Suppose $C$ and $D$ are as in Lemma \ref{isobasisp}.
If $R := -B^{-1} C \in \opn{det}B^{-1} M_2(\ZZ)$ and $R' \in \opn{det}B^{-1} M_2(\ZZ)$ satisfies
\begin{equation*}
  D - {}^tCB^{-1}C + {}^tR'A + {}^tAR' = 0,
\end{equation*}
then the required base change is given by the matrix
\begin{equation}\label{BaseChangeN}
  M
  =
  \begin{pmatrix} 
    I & 0 & R' \\ 
    0 & I & R \\ 
    0 & 0 & I 
  \end{pmatrix}.
\end{equation}
\end{proof}
\subsection{Counting boundary components}
We determine the $\opn{O}(L_{6, 2})$-orbits of primitive totally isotropic rank 2 sublattices of $L_{6,2}$ along the lines of \cite{scattone}.
As in \cite{scattone}, the $\opn{O}(L_{6,2})$-orbits of primitive isotropic vectors in $L_{6,2}$ can be determined in a straightforward manner by using the Eichler criterion (\S10 \cite{eichler}), and so we omit the calculation here.
\begin{lem}{(Lemma 4.1 of \cite{zbMATH03832089})}\label{Eperpdisc} 
 If $M$ is a non-degenerate even lattice, and $E \subset M$ is a primitive totally isotropic sublattice, then the discriminant form of the lattice $E^{\perp} / E$ is isomorphic to $H_E^{\perp}/H_E \subset D(M)$.
\end{lem}
\begin{lem}\label{HE}
If the rank 2 sublattice $E \subset L_{6,2}$ is primitive and totally isotropic, then
$E^{\perp}/E \cong \la -6 \ra \op \la -2 \ra$ or $E^{\perp}/E \cong A_2(-1)$.
\end{lem}
\begin{proof}
The lattice $E^{\perp}/E$ is negative definite and, by Lemma \ref{Eperpdisc}, $D(E^{\perp}/E) \cong H_E^{\perp}/H_E$.
Identify $D(L_{6,2})$ with $C_6 \op C_2$. 
If $(a,b) \in D(L_{6,2})$ is isotropic, then $a^2 / 6 + b^2 / 2 = 0 \bmod{2 \ZZ}$ and so, $(a,b)=(0,0)$ or $(a,b) = (3,1)$. 
If $H_E = \{ (0,0) \}$, then $H_E^{\perp}/H_E = D(L_{6,2})$ with discriminant form $((-1/6) \op (1/2), C_6 \op C_2)$. 
If $H_E = \la (3,1) \ra$, then $H_E^{\perp} = \la (1,1) \ra$ and $H_E^{\perp}/H_E \cong \la (2,0) \ra$ with discriminant form $( (1/3), C_3)$. 
By tables in \cite{SPLAG}, the two negative definite even lattices of determinant 12 are 
\begin{equation*}
\la -6 \ra \op \la -2 \ra
\end{equation*}
and 
\begin{equation}\label{lattice2}
\begin{pmatrix}
-4 & -2 \\ 
-2 & -4 
\end{pmatrix}.
\end{equation}
The discriminant form of \eqref{lattice2} is inequivalent to $( (1/2)^{\op 2} \op (-1/3), C_2^{\op 2} \op C_3)$. 
Therefore, if $H_E = \la (0,0) \ra$ then $E^{\perp} / E \cong \la -6 \ra \op \la -2 \ra$; 
if $H_E = \la (3,1) \ra$ then, from tables in \cite{SPLAG}, $E^{\perp}/E \cong A_2(-1)$.
\end{proof}
\begin{lem}\label{normalform}
If $E \subset L_{6,2}$ is a primitive totally isotropic sublattice of rank $2$, then there exists a $\ZZ$-basis $\{v_1, \ldots, v_6 \}$ of $L_{6,2}$ so that $\{v_1, v_2 \}$ is a basis for $E$, $\{v_1, \ldots, v_4\}$ is a basis for $E^{\perp} \subset L_{6,2}$, and the Gram matrix 
\begin{equation}\label{normalformeq}
Q = ((v_i, v_j)) = 
\begin{pmatrix} 
0 & 0 & P \\ 
0 & B & C \\ 
P & {}^tC & D 
\end{pmatrix}.
\end{equation}
Moreover,
\begin{enumerate}
\item if $H_E = \la (1,1) \ra$, then  $B = \la -6 \ra \op \la -2 \ra$, $C=D=0$, and  
\begin{equation*}
P = 
\begin{pmatrix}
0 & 1 \\ 
1 & 0 
\end{pmatrix};
\end{equation*}
\item if $H_E = \la (3,1) \ra$, then $B = A_2(-1)$,  
\begin{center}
\begin{tabular}{c c c}
$C = 
\begin{pmatrix}
0 & 0 \\ 
c & 0 
\end{pmatrix}$, 
&
$D = 
\begin{pmatrix}
2d & 0 \\ 
0 & 0 
\end{pmatrix}$,
&
$P = 
\begin{pmatrix}
0 & 1 \\ 
3 & 0 
\end{pmatrix}$,
\end{tabular}
\end{center}
for $c \in \{0,1,2\}$, $d \in \{0,1,2\}$.
\end{enumerate}
\end{lem}
\begin{proof}
As in Lemma \ref{isobasisp}, take a basis $\{v_1, \ldots, v_6 \}$ for $L_{6,2}$ so that $\{v_1, v_2\}$ is a basis for $E$, and $\{v_1, \ldots, v_4 \}$ is a basis for $E^{\perp}$. 
Suppose that
\begin{equation*}
Q = ((v_i, v_j)) = 
\begin{pmatrix} 
0 & 0 & A_0 \\ 
0 & B_0 & C_0 \\ 
A_0 & {}^tC_0 & D_0 
\end{pmatrix}.
\end{equation*}
By Lemma \ref{HE}, $H_E = \la (0,0) \ra$ or $H_E = \la (3,1) \ra$. 

If $H_E = \la (0,0) \ra$ then, because of the existence of a Smith normal form, there exist integral matrices $U$ and $Z$ so that 
\begin{equation*}
U A_0 Z = 
\begin{pmatrix} 
0 & 1 \\ 
1 & 0 
\end{pmatrix}.
\end{equation*}
By Lemma \ref{HE}, there exists $X \in \opn{GL}(2, \ZZ)$ so that ${}^tX B_0 X = B = \la -6 \ra \op \la -2 \ra$.
Therefore, the matrix $g_1:=\opn{diag}({}^tU, X, Z)$ transforms $Q$ to $Q'$ where
\begin{equation*}
Q' = {}^tg_1Qg_1 
=
\begin{pmatrix} 
0 & 0 & A \\ 
0 & B & C_1 \\ 
{}^t A & {}^t C_1 & D_1 
\end{pmatrix},
\end{equation*}
and
\begin{equation*}
A =
\begin{pmatrix}
0 & 1 \\
1 & 0 
\end{pmatrix}.
\end{equation*} 
The integral matrix $g_2$ defined by  
\begin{equation*}
g_2=
\begin{pmatrix} 
I & -{}^tA {}^t C_1 & 0 \\ 
0 & I & 0 \\ 
0 & 0 & I 
\end{pmatrix}
\end{equation*}
transforms $Q'$ to $Q''$, where
\begin{equation*}
Q'' = 
\begin{pmatrix} 
0 & 0 & A \\ 
0 & B & 0 \\ 
{}^t A & 0 & D_2 
\end{pmatrix}.
\end{equation*}

If $g_3$ is an integral matrix of the form
\begin{equation*}
g_3= 
\begin{pmatrix} 
I & 0 & W \\
0 & I & 0 \\ 
0 & 0 & I 
\end{pmatrix},
\end{equation*}
then $g_3$ sends $D_2 \mapsto D_2 + {}^tW A + {}^tAW$. 
One checks that the set $\{ {}^tWA + {}^tAW \mid W \in M_2(\ZZ) \}$ contains all matrices of the form 
\begin{equation*}
\begin{pmatrix} 
2a & b \\ 
b & 2c 
\end{pmatrix}
\end{equation*}
for $a,b,c \in \ZZ$. 
Therefore, there exists integral $W$ and $d_{11}'$, $d_{22}' \in \{0,1\}$ so that $d_{11}' \equiv d_{11} \bmod{2}$, $d_{22}' \equiv d_{22} \bmod{2}$, and 
\begin{equation*}
g_3: D_2 
\mapsto
\begin{pmatrix}
d_{11}' & 0 \\
0 & d_{22}' 
\end{pmatrix}.
\end{equation*}
As the form $Q$ is even, both $d_{11}$ and $d_{22}$ are even, and so there exists $W$ so that $g_3$ sends $D_2$ to $0$.
Therefore, the matrix $g_3 g_2 g_1$ gives the base change required in the statement of the theorem.

If $H_E = \la (3,1) \ra$ then, because of the Smith normal form, there exist $U, Z \in \opn{GL}(2, \ZZ)$ such that 
\begin{equation*}
U A_0 Z = 
\begin{pmatrix} 
0 & 1 \\ 
3 & 0 
\end{pmatrix}.
\end{equation*}
Moreover, there exists $X \in \opn{GL}(2, \ZZ)$ such that ${}^tX B_0 X = B = A_2(-1)$.
Therefore the matrix $g_4$ given by $g_4=\opn{diag}({}^t U, X, Z) \in \opn{GL}(2, \ZZ)$ transforms $Q$ to $Q'$ where
\begin{equation*}
Q' = 
{}^tg_1Qg_1 = 
\begin{pmatrix} 
0 & 0 & A \\ 
0 & B & C_1 \\ 
{}^t A & {}^t C_1 & D_1 
\end{pmatrix}
\end{equation*} 
and 
\begin{equation*}
A=
\begin{pmatrix} 
0 & 1 \\ 
3 & 0 
\end{pmatrix}.
\end{equation*}
Let $g_5$ be an integral matrix of the form 
\begin{equation*}
g_5=
\begin{pmatrix} 
I & S & 0 \\ 
0 & I & T \\ 
0 & 0 & I 
\end{pmatrix},
\end{equation*}
for some $S, T \in M_2(\ZZ)$. 
We claim that $(s_{ij}):=S$ and $(t_{ij}):=T$ can be chosen so that
\begin{equation*}
{}^t SA + BT + C_1 = 
\begin{pmatrix} 
0 & 0 \\ 
a & 0 
\end{pmatrix}
\end{equation*}
where $a$ is determined modulo 3. 
As 
\begin{equation}\label{normalformeq1}
{}^tSA + TQ + C_1= 
\begin{pmatrix} 3s_{21} - 2s_{11} - t_{21}  + c_{11} & s_{11} - 2t_{12} - t_{22} + c_{12} \\ 
3s_{22} - 2t_{21} - t_{11} + c_{21} & s_{12} - t_{12} - 2t_{22} + c_{22}
\end{pmatrix}, 
\end{equation}
then the claim about the second column is immediate, as $s_{11}$ and $t_{12}$ are both free. 

Let $\delta:=2t_{11} + t_{21}$ and $s_{21}=0$.
By taking $t_{11}$ and $t_{21}$ so that $\delta = -c_{11}$, the first column of \eqref{normalformeq1} can be mapped to ${}^t(0, 3s_{22} - c_{11} + c_{21})$.
Therefore, with an appropriate choice of $s_{22}$, the matrix $g_5$ transforms $Q'$ to
\begin{equation*}
Q'' = 
\begin{pmatrix} 
0 & 0 & A \\ 
0 & B & C_0 \\ 
{}^tA & {}^tC_0 & D_2 
\end{pmatrix},
\end{equation*}
where $C_0$ is as in the statement of the theorem. 

We next put $D_2$ in the correct form. 
If $g_6$ is an integral matrix of the form 
\begin{equation*}
g_6 = 
\begin{pmatrix} 
I & 0 & W \\ 
0 & I & 0 \\ 
0 & 0 & I 
\end{pmatrix}, 
\end{equation*}
 then $g_6$ sends 
\begin{equation*}
D_2 \mapsto D_2 + {}^tW A + {}^tAW.
\end{equation*}
One checks that the set $\{ {}^tWA + {}^tAW \mid W \in M_2(\ZZ)\}$ contains all matrices of the form 
\begin{equation*}
\begin{pmatrix} 
6a & b \\ 
b & 2c 
\end{pmatrix}
\end{equation*}
for $a,b,c \in \ZZ$. 
Therefore if $(d_{ij}):=D$, there exists $W$ so that  
\begin{equation*}
g_3: D_2 
\mapsto 
\begin{pmatrix} 
d_{11}' & 0 \\ 
0 & d_{22}' 
\end{pmatrix},
\end{equation*}
for $d_{11}' \in \{0, \ldots, 5\}$, 
$d_{22}' \in \{0,1\}$, 
and where $d_{11}' \equiv d_{11} \bmod{6}$ and $d_{22}' \equiv d_{22} \bmod{2}$.
As the form $Q$ is even, then both $d_{11}$ and $d_{22}$ are even.
Therefore, there exists $W$ so that $d_{11}'$ is even and $d_{22}'=0$.
Therefore, the matrix $g_6 g_5 g_4$ gives the base change required in the statement of the theorem.
\end{proof}
\begin{thm}\label{boundaryboundsthm}
For prime $p>3$, the modular variety $\fc_{2p^2}$ has at most $160(p^5+p^2)$ boundary curves.
\end{thm}
\begin{proof}
If $E_1, E_2 \subset L_{6,2}$ are totally isotropic primitive sublattices of rank 2 with the same normal form \eqref{normalformeq} then, by Lemma \ref{normalform}, there exists $g \in \opn{O}(L_{6,2})$ so that $g(E_1) = E_2$.
Therefore, by counting normal forms, there are at most 20 totally isotropic primitive rank 2 sublattices of $L_{6,2}$ up to $\opn{O}^+(L_{6,2})$-equivalence. 
By Theorem \ref{finindexthm},  
\begin{equation*}
 \vert \opn{O}^+(L_{6,2}) : \opn{O}^+(L, h_{2p^2}) \vert \leq 8 (p^5 + p^2),
\end{equation*} 
and so, up to $\opn{O}^+(L, h_{2p^2})$-equivalence, there are at most $160(p^5 + p^2)$ boundary curves.
\end{proof}
\subsection{Counting singularities}
In this section we count singularities in the boundary of a toroidal compactification of $\fc_{2p^2}$.
We shall only consider the compactification in a neighbourhood of a boundary curve. 
The structure of the boundary in a neighbourhood of a boundary point is different, and presents additional toric considerations.

Throughout, we assume the boundary component $F$ corresponds to a rank 2 primitive totally isotropic sublattice $E \subset L_{6, 2p^2}$ and define $N=a_1 a_2 \opn{det} B$, where $a_1$, $a_2$, $B$ are as in Lemma \ref{isobasisp}.
\begin{lem}\label{NWU}
On the basis given in Lemma \ref{Qbasis}, the groups $N(F)$, $W(F)$ and $U(F)$ are given by
\begin{align*}
& N(F) = \left \{ 
\begin{pmatrix} 
U & V & W \\ 
0 & X & Y \\ 
0 & 0 & Z 
\end{pmatrix} 
\mid
\begin{matrix} 
&\text{ ${}^t UAZ = A$, ${}^t XBX = B$,  ${}^t VAZ = 0$, ${}^t XBY + {}^t VAZ = 0$, } \\ 
&\text{ ${}^t YBY + {}^t ZAW + {}^t WAZ = 0$, $\opn{det}(U) > 0$},  
\end{matrix} 
\right\} \\
& W(F) = 
\left \{ 
\begin{pmatrix} 
I & V & W \\ 
0 & I & Y \\ 
0 & 0 & I 
\end{pmatrix} 
\mid 
\text{$BY + {}^t VA = 0$, ${}^t Y BY + AW + {}^t WA = 0$} \right \}, \\
& U(F) = \left \{ 
\begin{pmatrix} I & 0 & 
\begin{pmatrix} 0 & a_1a_2x \\ 
- x & 0 
\end{pmatrix} \\ 
0 & I & 0 \\ 
0 & 0 & I 
\end{pmatrix} 
\mid
 x \in \RR \right \}.
\end{align*}
Furthermore, if $g \in N(F)$, then $g \in N(F) \cap \opn{O}(L_{6,2p^2})$ if and only if 
\begin{equation*}
M^{-1} g M 
=
\begin{pmatrix}
U & V & -VB^{-1}C + W + UR' -R'Z \\
0 & X & Y - XB^{-1}C + B^{-1}CZ \\
0 & 0 & Z 
\end{pmatrix} \in \opn{GL}(6, \ZZ)
\end{equation*}
where $M$ is as in \eqref{BaseChangeN}.
\end{lem}
\begin{proof}
The first part follows from direct calculation (as in Section 2.12 of \cite{Kondo}).
The second part follows as in Proposition 2.27 of \cite{GHSK3}.
\end{proof}

As in \cite{Kondo}, $\dc_{L_{6,2p^2}}(F)$ can be identified with a Siegel domain of the third kind inside $\CC \times \CC^2 \times \HH^+$.
The identification proceeds by taking homogeneous coordinates $[t_1:\ldots:t_6]$ for $\PP(L_{6,2p^2} \otimes \CC)$ and mapping  $\dc_{L_{6,2p^2}}(F) \rightarrow \PP(L_{6,2p^2} \otimes \CC)$ by setting $t_6=1$ and
\begin{equation}\label{homogdef}
\begin{cases}
& t_1 \mapsto z \in \CC \\ 
& t_3 \mapsto w_1 \in \CC \\
& t_4 \mapsto w_2 \in \CC \\
& t_5 \mapsto \tau \in \HH^+ \\ 
& t_2 \mapsto \frac{  -2 a_1 z \tau - {}^t (w_1, w_2)B(w_1, w_2)}{2 a_1 a_2}.
\end{cases}
\end{equation}
\begin{prop}\label{Naction}
If 
\begin{equation*}
g = 
\begin{pmatrix} 
U & V & W \\ 
0 & X & Y \\ 
0 & 0 & Z 
\end{pmatrix} 
\in N(F)
\end{equation*}
is as in Lemma \ref{NWU}, 
where $Z = \begin{pmatrix} a & b \\ c & d \end{pmatrix}$, then the action of $g$ on $\dc_L(F)$ is given by
\begin{equation*}
\begin{cases}
& z \mapsto \frac{z}{\opn{det} Z} + (c \tau + d)^{-1} \left ( \frac{c}{2 a_1 \opn{det} Z } {}^t \underline{w} B \underline{w} + \underline{V_1} \underline{w} + W_{11} \tau + W_{12} \right ) \\
& \underline{w} \mapsto (c \tau + d)^{-1} \left ( X \underline{w} + Y \bsm \tau \\ 1 \esm \right ) \\
& \tau \mapsto \frac{a \tau + b} {c \tau +d }. 
\end{cases}
\end{equation*}
\end{prop}
\begin{proof}
As in \cite{GHSK3}.
\end{proof}
\begin{lem}\label{Zlemma}
Let 
\begin{equation}\label{Zdef}
Z= 
\begin{pmatrix} 
a & b \\ 
c & d  
\end{pmatrix}
\in \Gamma_N,
\end{equation}
where $\Gamma_N \subset \opn{SL}(2, \ZZ)$ is the principal congruence subgroup of level $N$. 
Then, on the basis given in Lemma \ref{Qbasis}, the map
\begin{equation*}
Z \mapsto g_Z 
= 
\begin{pmatrix} 
Z' & 0 & 0 \\ 
0 & I & 0 \\ 
0 & 0 & Z 
\end{pmatrix},
\end{equation*}
where
\begin{equation}\label{Z'def}
Z' = 
\begin{pmatrix} 
d & -ca_2 \\ 
- b / a_2 & a 
\end{pmatrix},
\end{equation}
defines an embedding of $\Gamma_N$ in $N(F) \cap \opn{O}(L_{6,2p^2})$.
\end{lem} 
\begin{proof}
Suppose $A$, $B$, $C$, and $R'$ are as in Lemma \ref{isobasisp}.
Let  $P:=Z'R'-R'Z$, $Q:=-B^{-1}C + B^{-1}CZ$, and let $Z$, $Z'$ be given by \eqref{Zdef}, \eqref{Z'def}, respectively.
As $Z'={}^t(A Z^{-1} A^{-1})$ then, by Lemma \ref{NWU}, $g_Z \in N(F) \cap \opn{O}(L_{6,2p^2})$ if and only if $Z'$, $P$, $Q \in M_2(\ZZ)$.

As 
\begin{equation*}
Z' =
\begin{pmatrix}
d & -c a_2 \\
-b/a_2 & a 
\end{pmatrix}
\end{equation*}
and $Z \in \Gamma_N$, then $Z' \in M_2(\ZZ)$.

If 
\begin{equation*}
R' = 
\begin{pmatrix}
w & x \\
y & z \\
\end{pmatrix},
\end{equation*}
then
\begin{equation*}
P = 
\begin{pmatrix}
-a_2 c y - aw + dw - cx & -a_2cz - bw \\
-cz - bw/a_2 & -by + az - dz - bx/a_2 
\end{pmatrix}.
\end{equation*}
As $Z \in \Gamma_N$, then $P \in M_2(\ZZ)$.

As $Z \equiv I \bmod{N}$, then $C - CZ \equiv 0 \bmod{N}$.
As $\opn{det}B \vert N$ and $R'$, $B^{-1} \in \opn{det} B^{-1} M_2(\ZZ)$,  then
\begin{equation*}
Q = B^{-1} (CZ - C) \in M_2(\ZZ),
\end{equation*}
and the result follows.
\end{proof}
\begin{lem}\label{Ylemma}
If $Y \in M_2(N \ZZ)$ then, on the basis given in Lemma \ref{Qbasis}, there exists $g_Y \in W(F) \cap \opn{O}(L_{6,2p^2})$ of the form
\begin{equation*}
g_Y = 
\begin{pmatrix}
I & * & * \\ 
0 & I & Y \\ 
0 & 0 & I
\end{pmatrix}.
\end{equation*}
\end{lem}
\begin{proof}
Fix $Y \in M_2(N \ZZ)$.
Let $A$, $B$, and $C$ be as in Lemma \ref{isobasisp} and $M$ be as in Lemma \ref{Qbasis}.
By Lemma \ref{NWU},
\begin{equation*}
g_Y
=
\begin{pmatrix}
I & V & W \\
0 & I & Y \\
0 & 0 & I 
\end{pmatrix}
\end{equation*}
belongs to $W(F)$ if and only if both
\begin{equation}\label{Y1}
BY + {}^t VA = 0
\end{equation}
and
\begin{equation}\label{Y2}
{}^t YBY + AW + {}^t WA = 0
\end{equation}
are satisfied.
If $g_Y \in W(F)$ then, by Lemma \ref{NWU}, $g_Y \in W(F) \cap \opn{O}(L_{6,2p^2})$ if and only if  
$V$, $Y$, and $P$ are integral matrices, where
\begin{equation*}
P:= W - VB^{-1}C.
\end{equation*}
Equation \eqref{Y1} has a solution in $V$ if $Y \in M_2(a_1 a_2 \ZZ)$, and the matrix $P$ is integral if $V \in M_2( (\opn{det} B) \ZZ)$.
Therefore, by \eqref{Y1}, both conditions are satisfied if $Y \in M_2(N \ZZ)$. 

If 
\begin{equation*}
W = 
\begin{pmatrix}
w_{11} & w_{12} \\
w_{21} & w_{22}
\end{pmatrix},
\end{equation*}
then \eqref{Y2} is equivalent to
\begin{align*}
- {}^tYBY 
& = AW + {}^tWA \\
& = 
\begin{pmatrix}
2a_1 w_{11} & a_1 w_{12} + a_1 a_2 w_{21} \\
a_1 w_{12} + a_1 a_2 w_{21}  &  2a_1 a_2 w_{22}
\end{pmatrix}, 
\end{align*}
and so has a solution in $W$ if $Y \in M_2(N \ZZ)$.  
\end{proof}
Suppose $(z, \underline{w}, \tau)$ are coordinates for $\dc_{L_{6,2p^2}}(F)$, as in \eqref{homogdef}.
Local coordinates for a cover $\dc_{2p^2}(F)^{tor}$ of a toroidal compactification $\fc^{tor}_{2p^2}$ in a neighbourhood of $F$ are obtained from $(z, \underline{w}, \tau)$ by replacing $z$ with $u = \opn{exp}_{a_1}(z) = e^{2 \pi i /a_1}$, and allowing $u=0$.
The group acting on the cover is given by $G(F):= (N(F) \cap \opn{O}^+(L, h)) / (U(F) \cap \opn{O}^+(L, h) )$.
\begin{thm}\label{fcFsingsthm}
If $g' \in G(F)$ fixes a point $x = (0, \underline{w}, \tau) \in \dc_{2p^2}(F)^{tor}$ then, in local coordinates around $x$, $g'$ acts by
\begin{equation}\label{gaction}
\opn{diag}(\omega_0, \xi \omega_1, \xi \omega_2, \xi^2),
\end{equation}
where $\omega_1$, $\omega_2$, and $\xi$ are 4-th or 6-th roots of unity, and $\omega_0$ is a 12-th root of unity.
For given $(\omega_1, \omega_2, \xi)$, bounds for the number of connected components of $\dc_{2p^2}(F)^{tor}$ fixed by some $g \in G(F)$ acting as in \eqref{gaction} are given in Table \ref{counttab1} and Table \ref{counttab2}, where $J_{N,B} = N \opn{det} B$ and $\omega = e^{2 \pi i /3}$.
Assuming $g'$ acts non-trivially, the image of $x$ in $\fc_{2p^2}^{tor}$ is singular.
For invariants $(a_1, a_2)$ corresponding to $F$, the values of $N$, $\opn{det} B$, and $K_N$ are given in Table \ref{invstab}.
\begin{table}[h!]
\begin{tabular}{|c|c|| c | c | c | c |}
\cline{3-6}
\multicolumn{2}{c}{} 
&
\multicolumn{4}{| c |}{$\xi$}\\
\cline{3-6}
\multicolumn{2}{c|}{} & $i$ & $-1$ & $-i$ & $1$ \\
\hline
\multirow{6}{1em}
{$\chi_X$} & $\Phi_1^2$ & $2^8 K_N J_{N,B}^4 $ & $1 $ & $2^8 K_N J_{N,B}^4 $ & -  \\
& $\Phi_1 \Phi_2$ & $2^8 K_N J_{N,B}^4 $ & $2^5 J_{N,B}$ & $2^8 K_N J_{N,B}^4 $ & $2^5 J_{N,B}$ \\
& $\Phi_2^2$ & $2^8 K_N J_{N,B}^4$ & - & $2^8 K_N J_{N,B}^4 $ & $1$ \\
& $\Phi_3$ & $2^4 K_N J_{N,B}^4 $ & $1$ & $2^4 K_N J_{N,B}^4$ &  $1$ \\
& $\Phi_4$ & $2^4 K_N J_{N,B}^2$ & $1$ & $2^4 K_N J_{N,B}^2$ & $1$ \\
& $\Phi_6$ & $2^4 K_N J_{N,B}^4$ & $1$ & $2^4 K_N J_{N,B}^4$ & $1$ \\
\hline
\end{tabular}\caption{}
\label{counttab1}
\end{table}
\begin{table}[h!]
\begin{tabular}{|c|c|| c | c | c | c |}
\cline{3-6}
\multicolumn{2}{c}{} 
&
\multicolumn{4}{| c |}{$\xi$}\\
\cline{3-6}
\multicolumn{2}{c|}{} & $-\omega$ & $\omega$ & $\omega^2$ & $-\omega^2$ \\
\hline
\multirow{6}{1em}
{$\chi_X$} & $\Phi_1^2$  & $2^4 K_N J_{N,B}^4 $ & $2^4.3^4 K_N J_{N,B}^4 $ & $2^4.3^4 K_N J_{N,B}^4 $ & $2^4 K_N J_{N,B}^4 $ \\
& $\Phi_1 \Phi_2$ &  $2^4 K_N J_{N,B}^4$ & $2^4 K_N J_{N,B}^4$ & $2^4 K_N J_{N,B}^4 $ & $2^4 K_N J_{N,B}^4 $\\
& $\Phi_2^2$ &  $2^4.3^4  K_N J_{N,B}^4$ & $2^4 K_N J_{N,B}^4 $ & $2^4 K_N J_{N,B}^4 $ & $2^4.3^4  K_N J_{N,B}^4 $\\
& $\Phi_3$ &  $ 2^8 K_N J_{N,B}^4$ & $2^4 K_N J_{N,B}^2$ & $2^4 K_N J_{N,B}^2$ & $2^8 K_N J_{N,B}^4$ \\
& $\Phi_4$ & $2^4 K_N J_{N,B}^4$ & $2^4 K_N J_{N,B}^4$ & $2^4 K_N J_{N,B}^4$ & $2^4 K_N J_{N,B}^4$ \\
& $\Phi_6$ & $2^4 K_N J_{N,B}^2$ & $2^8 K_N J_{N,B}^4$ & $2^8 K_N J_{N,B}^4$ & $2^4 K_N J_{N,B}^2$\\
\hline
\end{tabular}\caption{}
\label{counttab2}
\end{table}
\begin{table}[h!]
\begin{tabular}{|c|c|c|c|}
\hline
$(a_1, a_2)$ & $N$ & $\opn{det} B$ & $K_N$ \\
\hline
$(1,1)$ & $1$ & $12p^2$ & $12p^2$ \\
$(1,p)$ & $p$ & $12$ & $12p$ \\
$(1,2p)$ & $2p$ & $3$ & $6p$ \\ 
\hline
\end{tabular}
\caption{}\label{invstab}
\end{table}
\end{thm}
\begin{proof}
Throughout, we assume $g \in N(F)$ represents a finite order element of $G(F)$, and that $g$ fixes the point $ x = (0, \underline{w}, \tau) \in \dc(F)^{tor}$. 
By Corollary 2.29 of \cite{GHSK3}, no element of $G(F)$ acts as a quasi-reflection.
Therefore, if $g$ acts non-trivially, by a theorem of Chevalley \cite{Chevalley}, each point in the fixed locus of $G(F)$ on $\dc_{2p^2}(F)^{tor}$ is singular.

As in Lemma \ref{NWU}, let
\begin{equation*}
g = 
\begin{pmatrix}
U & V & W \\
0 & X & Y \\
0 & 0 & Z
\end{pmatrix}
\end{equation*}
where
\begin{equation*}
Z = 
\begin{pmatrix}
a & b \\
c & d \\
\end{pmatrix}, 
\end{equation*}
and let $\xi = (c \tau + d)^{-1}$, $T = (I - \xi X)$.
As $g \in G(F)$ is of finite order, then $U$, $X$, $Z$ are of finite order.
By considering rational representations, $o(U), o(X)$, $o(Z) \in \{1,2,3,4,6\}$ and so the basis of Lemma \ref{HE} can be chosen so that each of $U$, $X$, $Z$ can be represented by one of
\begin{center}
\begin{tabular}{c c c c c c}
$\pm 
\begin{pmatrix}
1 & 0 \\
0 & 1 
\end{pmatrix}$,
& 
$\begin{pmatrix}
-1 & 0 \\
0 & 1 
\end{pmatrix}$,
&
$\begin{pmatrix}
0 & 1 \\
-1 & -1 
\end{pmatrix}$,
&
$\begin{pmatrix}
0 & -1 \\
1 & 0
\end{pmatrix}$,
& 
$\begin{pmatrix}
0 & -1 \\
1 & 1
\end{pmatrix}$.
\end{tabular}
\end{center}
Indeed, as $Z \in M_2(\ZZ)$ (Lemma \ref{NWU}) and $Z$ acts on $\HH^+$, then $Z \in \opn{SL}(2, \ZZ)$ (as noted in \cite{GHSK3}).

Suppose $o(Z) \in \{3,4, 6\}$.
By standard results on the elliptic elements of $\opn{SL}(2, \ZZ)$ \cite{Diamond}, $\tau$ is $\opn{SL}(2, \ZZ)$-equivalent to $i$, $\omega$ or $\omega^2$.
If $G_k$ is the Eisenstein series of weight $k$, then $G_4(i) \neq 0$, and $G_6(\omega) \neq 0$ (p. 10 \cite{Diamond}), and so $\xi$ is a 4-th root of unity if $Z$ is of order 4, and a 6-th root of unity if $Z$ is of order 3 or 6. 
If $c \tau + d = \pm 1$ then, as both $\{1, \omega \}$ and $\{1, i\}$ are linearly independent over $\QQ$,  
\begin{equation*}
Z = 
\begin{pmatrix}
\pm 1 & b \\
0 & \pm 1
\end{pmatrix}
\end{equation*}
which implies the contradiction $o(Z) = 1$ or $2$.
Therefore, 
\begin{equation*}
\la 1, \tau \ra 
= 
\begin{cases}
\la 1, i \ra & \text{if $o(Z)=4$}\\
\la 1, \omega \ra & \text{ if $o(Z) = 3$ or $6$.}
\end{cases}
\end{equation*}
The values of $\opn{det} T$ against $\chi_X$ and $\xi$ are given in Table \ref{thetab}.
\begin{table}[h!]
\begin{tabular}{|c | c || c | c | c | c | c | c | c | c |}
\cline{3-10}
\multicolumn{2}{c}{} 
&
\multicolumn{8}{| c |}{$\xi$}\\
\cline{3-10}
\multicolumn{2}{c|}{} & $i$ & $-1$ & $-i$ & $1$ & $-\omega$ & $\omega$ & $\omega^2$ & $-\omega^2$ \\
\cline{1-10}
\multirow{6}{1em}
{$\chi_X$} & $\Phi_1^2$ & $-2i$ & $4$ & $2i$ & 0 & $\omega$ & $-3\omega$ & $-3 \omega^2$ & $\omega^2$ \\
& $\Phi_1 \Phi_2$ & $2$ & $0$ & $2$ & $0$ & $1 - \omega^2$ & $1 - \omega^2$ & $1 - \omega$ & $1 - \omega$\\
& $\Phi_2^2$ & $2i$ & $0$ & $-2i$ & $4$ & $-3 \omega$ & $\omega$ & $\omega^2$ & $-3 \omega^2$\\
& $\Phi_3$ & $i$ & $1$ & $-i$ &  $3$ & $-2 \omega$ & $0$ & $0$ & $-2 \omega$ \\
& $\Phi_4$ & $0$ & $2$ & $0$ & $2$ & $- \omega$ & $-\omega$ & $-\omega^2$ & $-\omega^2$ \\
& $\Phi_6$ & $-i$ & $3$ & $i$ & $1$ & $0$ & $-2 \omega$ & $-2\omega^2$ & $0$\\
\hline
\end{tabular}\caption{}
\label{thetab}
\end{table}
By \eqref{homogdef},
\begin{equation*}
T \underline{w} = 
Y \bsm \tau \\ 1 \esm 
\subset 
\frac{\la 1, \tau \ra}{\opn{det} B} \times \frac{\la 1, \tau \ra}{\opn{det} B },
\end{equation*}
and so if $\opn{det} T \neq 0$, by Table \ref{thetab}, 
\begin{equation*}
\underline{w} \in
\frac{\la 1, \tau \ra}{K \opn{det} B}
\times 
\frac{\la 1, \tau \ra}{K \opn{det} B},
\end{equation*}
for appropriate $K \in \{1,2,3\}$.
Therefore, by using elements of the form $g_Y$ defined in Lemma \ref{Ylemma}, $\underline{w}$ can be reduced to one of $(K N \opn{det} B)^4$ points modulo $N(F) \cap \opn{O}(L_{6,2p^2})$.

If $o(Z) = 1$ or $2$,  then $\tau \in \HH^+$ is free and $\xi = \pm 1$.
By Lemma \ref{NWU}, $V=0$, $Y=0$, and so 
\begin{equation*}
T \underline{w} = \underline{0}.
\end{equation*}
If $\opn{det} T \neq 0$ then $\underline{w}=0$.
We consider each of the cases $\opn{det} T = 0$ separately.
If $(\chi_X, \xi) = (\Phi_1^2,1)$ or $(\chi_X, \xi)=(\Phi_2^2, -1)$, then $g$ acts as the identity (as it cannot act as a quasi-reflection); 
if $(\chi_X, \xi) = (\Phi_1 \Phi_2, 1)$, then $o(Z)=1$, $\tau$ is free, $w_1 \in (2 \opn{det} B)^{-1} \ZZ$, and $\underline{w}$ can be reduced to one of $2 N \opn{det} B$ lines by using elements of the form $g_Y \in N(F) \cap \opn{O}(L_{6,2p^2})$. 
(The case $(\Phi_1 \Phi_2, -1)$ proceeds as for $(\Phi_1 \Phi_2, 1)$.)
The remaining cases proceed as for $(\chi_X, \xi) = (\Phi_4, -i)$.
If $(\chi_X, \xi) = (\Phi_4, -i)$, then $w_1 - i w_2 \in (\opn{det} B)^{-1} \la 1, i \ra$  and $\underline{w}$ can be reduced to one of $(N \opn{det} B)^2$ lines using elements of the form $g_Y \in N(F) \cap \opn{O}(L_{6,2p^2})$.

By standard results on congruence subgroups (p.13 \cite{Diamond}), the index $K_N := \vert \opn{SL}(2, \ZZ) : \Gamma_N \vert$ is given by
\begin{equation*}
K_N = N^3 \prod_{p \mid N} \left (1 - \frac{1}{p^2} \right ).
\end{equation*}
Therefore, if $o(Z) = 3,4,6$, then $\tau$ can be reduced to one of $K_N$ cases modulo $N(F) \cap \opn{O}(L_{6,2p^2})$ by using elements of the form $g_Z \in \Gamma_N \subset N(F)_{\ZZ}$ defined in Lemma \ref{Zlemma}.

By Proposition \ref{modgroup}, $\widetilde{\opn{O}}^+(L_{6,2p^2}) \subset \opn{O}^+(L, h)$ and one obtains a final bound from \eqref{stabindex} by noting that $\vert \opn{O}^+(L, h) : \opn{O}(L_{6,2p^2}) \vert \leq 2^4$.

The statement about the action of $g$ in local coordinates follows as in \cite{Kondo}, and the values in Table \ref{invstab} can be calculated from Lemma \ref{isobasisp}.
The statement about $\omega_0$ follows by noting that $o(\omega_0)$ divides $lcm(o(U), o(X), o(Z))$. 
\end{proof}
\begin{rmk}
If $g \in G(F)$ acts as $\opn{diag}(\omega_0, \xi \omega_1, \xi \omega_2 \xi^2)$ in a neighbourhood of $(0, \underline{w}, \tau)$ and $\omega_0 \neq 1$, then $(0, \underline{w}, \tau)$ is not contained in the closure of the singular locus of $\fc_{2p^2}$.
The singular locus of $\fc_{2p^2}$ can be studied by different methods, as in \cite{daweskod}.
\end{rmk}
\section{Acknowledgements}
This paper originates from my PhD thesis.
I thank Professor G. K. Sankaran for his supervision, and the University of Bath for financial support in the form of a research studentship.
I also thank the Riemann Center for Geometry and Physics for providing excellent working conditions as I edited the paper, and for financial support in the form of a Riemann fellowship. 
\bibliographystyle{alpha}
\bibliography{bib}{}
\bigskip
{
Riemann Center for Geometry and Physics \\
Leibniz Universit\"at Hannover \\
Appelstra\ss e 2 \\
30167 Hannover  \\
Deutschland 
}
\bigskip
\\
\texttt{matthew.r.dawes@bath.edu}
\end{document}